\documentclass[10pt]{amsart}
\usepackage{amssymb}
\usepackage{epsfig}
\usepackage{url}
\usepackage{setspace}
\usepackage{pdflscape}
\theoremstyle{plain}

\newtheorem{thm}{Theorem}[section]

\newtheorem{rem}[thm]{Remark}

\newtheorem{conj}[thm]{Conjecture}

\newtheorem{prob}[thm]{Problem}
\def\cal{\mathcal}
\def\bbb{\mathbb}
\def\op{\operatorname}
\renewcommand{\phi}{\varphi}

\newcommand{\N}{\bbb{N}}
\newcommand{\Z}{\bbb{Z}}
\newcommand{\Q}{\bbb{Q}}

\begin{document}

\title[Some experiments with Ramanujan-Nagell type equations]{Some experiments with Ramanujan-Nagell type Diophantine equations}
\author{Maciej Ulas}

\thanks{The research of the author was supported by the grant of the Polish National Science Centre no. UMO-2012/07/E/ST1/00185}

\begin{abstract}
Stiller proved that the Diophantine equation $x^2+119=15\cdot 2^{n}$ has exactly six solutions in positive integers. Motivated by this result we are interested in constructions of Diophantine equations of Ramanujan-Nagell type $x^2=Ak^{n}+B$ with many solutions. Here, $A,B\in\Z$ (thus $A, B$ are not necessarily positive) and $k\in\Z_{\geq 2}$ are given integers. In particular, we prove that for each $k$ there exists an infinite set $\cal{S}$ containing pairs of integers $(A, B)$ such that for each $(A,B)\in \cal{S}$ we have $\gcd(A,B)$ is square-free and the Diophantine equation $x^2=Ak^n+B$ has at least four solutions in positive integers. Moreover, we construct several Diophantine equations of the form $x^2=Ak^n+B$ with $k>2$, each containing five solutions in non-negative integers.
We also find new examples of equations $x^2=A2^{n}+B$ having six solutions in positive integers, e.g. the following Diophantine equations has exactly six solutions:
\begin{equation*}
\begin{array}{ll}
  x^2= 57\cdot 2^{n}+117440512 & n=0 , 14 ,  16, 20, 24, 25, \\
  x^2= 165\cdot 2^{n}+26404 & n=0 , 5 ,  7, 8, 10, 12.\\
\end{array}
\end{equation*}
Moreover, based on an extensive numerical calculations we state several conjectures on the number of solutions of certain parametric families of the Diophantine equations of Ramanujan-Nagell type.
\end{abstract}

\maketitle

\section{Introduction}\label{sec1}
Let us consider the Diophantine equation $x^2=2^{n}-7$ and ask about its solutions in integers $(x,n)$. The question concerning characterization of all integral solutions of this equation was posed by Ramanujan in 1913. Ljunggren posed the same question in 1943. Finally, in 1948 Nagell found all solutions of this equation and since then this equation is known as Ramanujan-Nagell equation. It is nowdays well known that it has exactly five solutions $(x, n)=(1, 3), (3, 4), (5, 5), (11, 7)$ and  $(181, 15)$. Note that the proof of this fact appeared in English in 1960, see \cite{Na}. This result was also independently obtained by several writers (Lewis, Chowla, Browkin-Schinzel). This discovery was a good motivation for  mathematicians to work on more general equations of Ramanujan-Nagell type, i.e. the equations of the form
\begin{equation}\label{geneq}
x^2=Ak^{n}+B, \quad k\in\Z_{\geq 2}, A, B\in\Z\setminus\{0\}.
\end{equation}
Here we assume that $A, B$ are not both negative. From the Siegel result on prime divisors of polynomial values we know that this equation has only finitely many solutions in integers. Many papers are devoted to the study of integral solutions of the equation (\ref{geneq}) in the case when $A=1$ and $k$ is a prime number \cite{Ape, Bo, Bo1, Tza}.
In particular, it is known that if $A=1, B<0$ and $k$ is an odd prime number not dividing $B$, then the Diophantine equation (\ref{geneq}) has at most one solution in positive integers $x$ and $n$, unless $(p, B)=(3,2)$ or $(p, B)=(4a^2+1, 3a^2+1)$
for some $a\in \N$. In these cases, there are precisely two such solutions \cite{Ape, BuS}. In the case $A=1, B>0$ with an odd prime number $k$ satisfying $p\nmid B$, the equation (\ref{geneq}) has at most four solutions in positive integers $x, n$ \cite{Bo1}.

More recent results on the equation (\ref{geneq}) can be found in \cite{Ben}, where the authors are interested with solutions satisfying $n\geq 2$ (however, later in that paper they assume $n\geq 1$). One can also look into an interesting paper \cite{Sar} where the authors look for solutions of (\ref{geneq}) with $n\geq 2$. Typical result proved in the cited papers state that there are at most one, two or three solutions in positive integers $x, n$. More precisely, we do not know much about Ramanujan-Nagell equations which have many solutions and this motivated us to state the general problem of constructing Ramanujan-Nagell type equations with many solutions in non-negative integers.

\begin{prob}\label{mainprob}
Find examples of Ramanujan-Nagell type Diophantine equations with many solutions in non-negative integers.
\end{prob}

We are thus interested in the solutions of (\ref{geneq}) which are non-negative integers. We are aware that the statement of the problem above is not very precise. Here, the meaning of ``many" depends on the type of equation under consideration. For example, Stiller proved that the Diophantine equation $x^2=15\cdot 2^{n}-119$ has exactly six solutions in non-negative integers and noted that the Diophantine equation $x^2=35\cdot 2^{n}-391$ has exactly five solutions. So, in the case of $k=2$, by many we will understand five or six. Next, if $k>3$ then according to our best knowledge there is no example of a Ramanujan-Nagell equation which has five or more solutions. Thus, in this case by many we will understand five (or more). Similarly, if $A=\pm 1$ then by many we will understand three (or more).

Let us describe the content of the paper in some detail. In section \ref{sec2} we prove that for each $k\in\Z_{\geq 2}$ there are infinitely many pairs of integers $A,B$ such that $\gcd(A,B)$ is a square-free integer and the Diophantine equation $x^2=Ak^{n}+B$ has at least four solutions in non-negative integers. Next, we will describe our general search procedure which was used in our computer experiments. In particular we present two infinite families of the equations of the form $x^2=A2^{n}+B$ which have at least five solutions in integers. We also present new examples of this form with exactly six solutions. We also present some results concerning the equation (\ref{geneq}) with $k\geq 3$ having at least five solutions. In section \ref{sec3} we consider the equation $x^2=k^n+B$. For each even $k$ we construct an infinite family of $B$'s such that the equation $x^2=k^{n}+B$ has at least three solutions. We also raise the question concerning the construction of equations of the considered type which satisfy additionally the condition $\op{gcd}(k,B)=1$ with even $k$. This question is motivated by the observation made by Beukers in \cite{Bo}. In particular we found three equations of the form $x^2=k^{n}+B$ satisfying the condition $\op{gcd}(k,B)=1$ and having three solutions. Finally, in the last section we present some remarks concerning the equation (\ref{geneq}) with $A<0$.

\section{The method of search and the equation $x^2=Ak^n+B$ with $A>0$}\label{sec2}

We start with the following simple results which shows that for any $k\in \Z_{\geq 2}$ there are infinitely many (essentially) different values of $A, B$ such that the Diophantine equation $x^2=Ak^{n}+B$ has at least four solutions. This simple result is an immediate corollary from our earlier results presented in \cite{DoUl}.

\begin{thm}\label{fourvalues}
For each $k\in \Z\setminus\{-1,0,1\}$ there are infinitely many pairs of integers $A,B$ such that $\gcd(A,B)$ is square-free  and the Diophantine equation $x^2=Ak^{n}+B$ has at least four solutions in non-negative integers.
\end{thm}
\begin{proof}
In order to prove this result we use the polynomial $F(x)=Px+Q$, where
\begin{equation*}
P=8 \left(\sigma _1^3-4 \sigma _2 \sigma _1+8 \sigma _3\right),\quad\quad Q=\sigma _1^4-8 \sigma _2 \sigma _1^2+16 \sigma _2^2-64 \sigma _4,
\end{equation*}
and $\sigma_{i}=\sigma_{i}(a,b,c,d)$ is the $i$-th elementary symmetric polynomial for $i=1,2,3,4$. This polynomial was constructed in \cite[Corollary 2.9]{DoUl}, and satisfies the equalities $F(a)=\square,  F(b)=\square, F(c)=\square, F(d)=\square$. Here $a,b,c,d$ are pairwise distinct integer parameters  and $\square$ denotes a square of an integer. By taking now $(a,b,c,d)=(1,k^{p},k^{q},k^{r})$ we see that the Diophantine equation $x^2=Ak^{n}+B$ has solutions with $n=0,p,q,r$. Here, we denote by $A=A(p,q,r)$ and $B=B(p,q,r)$ the values of $P(1,k^{p},k^{q},k^{r})$ and $Q(1,k^{p},k^{q},k^{r})$, respectively. In order to finish the proof we need to show that for infinitely many triples of positive integers $p,q,r$ with $p<q<r$ we get infinitely many different pairs $(\op{sf}(A(p,q,r)),\op{sf}(B(p,q,r)))$, where $\op{sf}(N)$ is square-free part of an integer $N$. In other words, it is enough to show that for any fixed triple of integers $(p',q',r')$ there are infinitely many triples $(p,q,r)$ satisfying the condition $p<q<r$, such that the system of Diophantine equations
\begin{equation}\label{sysfour}
aT^2=A(p,q,r),\quad bT^2=B(p,q,r),
\end{equation}
has no integer solutions in $T$. Here, in order to shorten the notation, we put $a=\op{sf}(A(p',q',r')), b=\op{sf}(B(p',q',r'))$. In order to find suitable values of $p,q,r$ we first put $p=mt, q=mt+1, r=mt+2$, where $m, t$ are positive integers which will be chosen later. Using the expression for $A$ given above, we have
\begin{equation*}
A(p,p+1,p+2)=8((k^2-k-1)x^{m}+1)((k^2-k+1)x^{m}-1)((k^2+k-1)x^{m}-1),
\end{equation*}
where we put $x=k^{t}$. Let us denote the right hand side of the above equality by $H_{m}(x)$. We observe now that the (hyperelliptic) curve $C:\;aT^2=H_{m}(x)$ (treated as a curve in the $(x,T)$ plane) is of genus $g(C)=\lfloor\frac{3m}{2}\rfloor$. In particular, $g(C)\geq 2$ for $m\geq 2$. Using now the Faltings theorem \cite{Fal} we get that the curve $C$ has only finitely many rational solutions, say $(x_{1},T_{1}),\ldots (x_{n},T_{n})$. However, there are only finitely many values of $j\in\{1,\ldots,n\}$ such that $x_{j}=k^{t}$. If $t'$ is the biggest solution of the equation $x_{j}=k^{t}$ then for each $t>t'$ we immediately deduce that there are no integers satisfying the equality $H_{m}(k^{t})=aT^{2}$. Summing up we proved that for any given triple $(p',q',r')$ there are infinitely many values of $(p,q,r)$ of the form $(mt, mt+1, mt+2)$ with $m\geq 2$ and sufficiently large $t$ (dependent on the finite number of rational points on the curve $C$) such that the system (\ref{sysfour}) has no solutions in integers.
\end{proof}

\begin{rem}
{\rm We expect that much more is true. To be more precise: we think that for each triple $(p',q',r')$ of nonnegative integers with $p'<q'<r'$ there are only finitely many triples $(p,q,r)$ such that the system of Diophantine equations given by (\ref{sysfour}) has a non-trivial integer solution in $T$.
}
\end{rem}
We give now a short description of the method which allows us to find examples of the equations $x^2=Ak^n+B$ with various constraints on $k, A, B$ and many solutions. First, we consider the case of $A, B$ with $A$ positive. Given $A$, the idea is to choose an integer $B$ in such a way that the equation $x^2=Ak^n+B$ has two fixed solutions. In particular, because we also include the case when $k\mid A$ we can assume that the first solution is given by $n=0$. This implies that $B=x_{1}^2-A$ for certain $A, x_{1}\in\Z$. Next, if we want to have second solution with fixed $n=p$, where $p$ is a given positive integer, then necessarily $x_{2}^2=Ak^{p}+B$ for certain $x_{2}\in \Z$ and $x_{1}<x_{2}$. This implies that the numbers $x_{1}, x_{2}, A$ satisfy the equation $x_{2}^2-x_{1}^2=A(k^{p}-1)$. So, for any given $A$ we consider the set
\begin{equation*}
\cal{D}:=\{d\in\N:\;d\mid A(k^{p}-1)\},
\end{equation*}
i.e. the set of divisors of the number $K:=A(k^{p}-1)$. We thus see that for each $d\in \cal{D}$ with $d\leq K/2$ we have a solution of the equation  $x_{2}^2-x_{1}^2=K=d\cdot\frac{K}{d}$ given by
\begin{equation*}
x_{1}=\frac{1}{2}\left(\frac{K}{d}-d\right),\quad x_{2}=\frac{1}{2}\left(\frac{K}{d}+d\right)
\end{equation*}
and the corresponding $B$ we are looking for takes the form
\begin{equation*}
B=\left(\frac{1}{2}\left(\frac{K}{d}-d\right)\right)^{2}-Ak^{p}.
\end{equation*}
We thus see that for any given $d\in \cal{D}$ such that the value of $x_{1}$ is an integer we have values of $A, B$ which gives an equation $x^2=Ak^{n}+B$ with at least two solutions in positive integers (i.e. those with $n=0, p$). Moreover, in order to reduce the number of possible solutions without loss of generality we can assume that the number $\op{gcd}(A,B)$ is square-free. Performing then a brute force search of additional solutions of this equation we can find those values of $A, B$ which lead to equations with many solutions. Let us note that modifications of this method were used in investigations devoted to the study of Brocard-Ramanujan type equations \cite{Ul} and its various generalizations \cite{DoUl}.

We performed independent searches for the case of $k=2$ and $k>2$.  In the case of $k=2$ we found some infinite families of equations which have at least five solutions.

\begin{thm}
For each positive integer $m$ the Diophantine equations
\begin{align*}
  &x^2=(2^{3m}+1)2^{n}+1-2^{3m+3}, \\
  &x^2=\frac{1}{9}(2^{6m}-1)2^{n}+\frac{1}{9}(2^{6m+3}+1)
\end{align*}
have at least five solutions in positive integers. The first equation has solutions given by
\begin{align*}
(x,n)=&(3,3), (2^{2m+1}-2^{m+1}-1,m+2), (2^{3m+1}-1, 3m+2), \\
      &(2^{3m+2}+1,3m+4), (2^{6m+3}+2^{3m+2}-1, 9m+6).
\end{align*}
The second equation has solutions given by
\begin{align*}
(x,n)=&(2^{3m},0), \Big(\frac{1}{3}(2^{4m+1}+2^{2m+1}-1), 2m+2\Big), \Big(\frac{1}{3}(2^{6m+1}+1), 6m+2\Big),\\
      &\Big(\frac{1}{3}(2^{6m+2}-1), 6m+4\Big), \Big(\frac{1}{3}(2^{3(4m+1)}-2^{2(3m+1)}-1), 18m+6\Big).
\end{align*}                   \\
\end{thm}
\begin{proof}
We left the simple check that the displayed solutions satisfy corresponding equations to the reader. However, let us note that the Ramanujan-Nagell equation $x^2+7=2^n$ is contained in the first family (after division of both sides by 9) for $m=1$ with additional solution for $n=4$ (however, $3=m+2$ in this case).
\end{proof}

We also found new equations with six solutions extending the result of Stiller from \cite{Sti}. Our computations based on the method described above, were performed for $0<p\leq 30$ and $A\leq 10^{6}$. In this range we were able to find only two pairs $(A, B)$ such that the Diophantine equation $x^2=A2^{n}+B$ has six solutions. In both cases $B$ is even and positive. More precisely, we have the following result:

\begin{thm}\label{sixsolutions}
The Diophantine equation
$$
x^2= 57\cdot 2^{n}+117440512
$$
has exactly six solutions in non-negative integers, given by
$$
(x,n)=(10837, 0), (10880,14), (11008,16), (13312,20), (32768,24), (45056,25).
$$
Similarly, the Diophantine equation
$$
  x^2= 165\cdot 2^{n}+26404
$$
has exactly six solutions in non-negative integers, given by
$$
(x,n)=(163,0), (178,5), (218,7),  (262,8), (442,10), (838,12).
$$
\end{thm}
\begin{proof}
We start with the first equation. By computer search one can easily check that the only solution with $n\leq 26$ is the one displayed in the statement of the theorem. We thus can assume that $n\geq 26$. Let us put $n=26+m$ with $m\geq 0$ and let us observe that $117440512 =2^{24}\cdot 7$. By putting now $x=2^{12}X$ we reduce our problem to the proof that the Diophantine equation $X^2=4\cdot 57\cdot 2^{m}+7$ has no solutions in integers with $m\geq 0$. However, this is very easy because the shape of the equation immediately implies that $X^2\equiv 3\pmod{4}$ which is clearly impossible.

Let us consider the second equation. We observe that in order to find solutions it is enough to find all integer solutions of the Diophantine equations $x^2=ay^3+26404$ with $a\in\cal{A}$, where $\cal{A}=\{165, 2\cdot165, 4\cdot 165\}$. Then solutions in $y$ which are powers of 2 will correspond to the solutions of our equation. Equivalently, we are interested in the integer solutions of the Diophantine equations $x^2=Y^3+26404a^2$ with $Y$ coordinate of the form $Y=a2^{m}$, where $a\in\cal{A}$. Using the Magma computational package \cite{Mag} we were able to find all integer solutions of the Diophantine equations we are interested in. More precisely, we used {\tt IntegralPoints} procedure implemented in Magma which allows to find all integral points on given elliptic curve given by a Weierstrass equation. We gather the results of these computations in the table below.
\begin{center}
\begin{equation*}
\begin{array}{|l|l|}
  \hline
  a          & \mbox{Integral solutions}\; (Y, \pm x)\; \mbox{of}\; x^2=Y^3+26404a^2  \\
  \hline
  165          & (-780, 15630), (-264, 26466), (100, 26830), (165, 26895), \\
               & (1485, 63195), (2640, 138270) \\
  2\cdot 165 &(-824, 48124), (1320, 71940), (2640, 145860), (23265, 3548985)\\
  4\cdot 165 &(-2196, 30192), (-1440, 92280), (1320, 117480), (2640, 172920), \\
               &(3025, 197945), (5896, 465256), (29880, 5166120), (44440, 9368920) \\
   \hline
\end{array}
\end{equation*}
\end{center}
\begin{center}
Table 1.
\end{center}

A quick inspection of the corresponding solution sets reveals that the only non-negative integer solutions of $x^2=165\cdot 2^{n}+26404$ satisfy $n\in\{0,5,7,8,10,12\}$.

\end{proof}

\begin{rem}\label{rem1}
{\rm It is possible to give a completely elementary proof of the fact that the only solutions of the equation $x^2=165\cdot 2^{n}+26404$ correspond with $n=0,5,7,8,10,12$ using the approach presented in Stiller paper \cite{Sti} (see also \cite{Has, Tza}). In this case we need to consider solutions of two Pell type equations: $x^2=165y^2+26404$ and $x^2=2\cdot 165y^2+26404$. The set of integer solutions of the first equation is parameterized by the sequence $(x_{n}^{(1)},y_{n}^{(1)})$ satisfying the recurrence relations
\begin{equation*}
x_{n+1}^{(1)} = 1079x_{n}^{(1)} + 13860y_{n}^{(1)}, \quad y_{n+1}^{(1)} = 84x_{n}^{(1)} + 1079y_{n}^{(1)},
\end{equation*}
where the initial values $(x_{0}^{(1)},y_{0}^{(1)})$ belong to the set $V_{1}$, where
\begin{align*}
V_{1}=\{&(\pm 163,1), (\pm 167,3), (\pm 233,13), (\pm 262,16), (\pm 383, 27), (\pm 442,32), (\pm 838,64),\\
        &(\pm 977,75), (\pm 1142, 82), (\pm 1333, 103), (\pm 2587, 201), (\pm 3023, 235)\}.
\end{align*}

The set of integer solutions of the second equation is parameterized by the sequence $(x_{n}^{(2)},y_{n}^{(2)})$ satisfying the recurrence relations
\begin{equation*}
x_{n+1}^{(2)} = 109x_{n}^{(2)} + 1980y_{n}^{(2)}, \quad y_{n+1}^{(2)} = 6x_{n}^{(2)} + 109y_{n}^{(2)},
\end{equation*}
where the initial values $(x_{0}^{(2)},y_{0}^{(2)})$ belong to the set $V_{2}$, where
\begin{equation*}
V_{2}=\{(\pm 178, 4), (\pm 218, 8), (\pm 398, 20), (\pm 1102, 60)\}.
\end{equation*}
Performing now the careful (and very tedious) analysis of the sequence(s) $y_{n}^{(i)}\pmod{2^{k}}$ for $k=1,2,\ldots,7$ and the sequence $y_{n}^{(i)}\pmod{m_{i}}$ with suitable chosen $m_{i}$ (depending on the initial values from the set $V_{i}$), one can find that all non-negative integer solutions of our initial equation $x^2=2\cdot 165\cdot2^{n}+26404$ corresponding to $n=0,5,7,8,10,12$.
}
\end{rem}

\begin{rem}{\rm From the work of Beukers \cite{Bo,Bo1} it is known that the only equation of the form $x^2=2^n+B, B\equiv 1\pmod{2}$, which has at least five solutions satisfying $n\geq 1$, corresponds to $B=-7$. However, if we drop the condition on $B$ and allow non-negative solutions, we have the equation $x^2=2^n+1088$ with five solutions given by
$$
(x,n)=(33,0), (40,9), (56,11), (72,12), (184,15).
$$
Here we have $1088=2^{6}\cdot 17$. However, we were unable to find more equations with this property.
}
\end{rem}

Remarkably, in case of $k>2$ which is not a power of two we have found only a few equations of the form $x^2=Ak^{n}+B$ (satisfying the condition $\op{gcd}(A,B)$ is a square-free number), which have at least five solutions. We gather the result of our computation in the following theorem.

\begin{thm}
Each of the following Diophantine equations has exactly five solutions in non-negative integers $(x,n)$:
\begin{center}
\begin{equation*}
\begin{array}{l|lll}
\mbox{No.}&         \mbox{The equation}        & \mbox{The solution set for}\;n\\
\hline
 1 &  x^2=28\cdot 3^{n}+2997,        & n=0, 2, 5, 6, 10,                  \\
 2 &  x^2=70\cdot 3^{n}+414,         & n=0, 3, 4, 5, 8,                    \\
 3 &  x^2=130\cdot 3^{n}+5550606,    & n=0, 6, 11, 15, 16,                  \\
 4 &  x^2=148\cdot 3^{n}+41877,      & n=0, 5, 6, 9, 17,                   \\
 5 &  x^2=8740\cdot 3^{n}+57402189,  & n=0 , 4 ,  9, 15, 29 ,             \\
 6 &  x^2=6\cdot 5^{n}+11875,        & n=0, 4, 5, 6, 9,                      \\
 7 &  x^2=14\cdot 5^{n}+6875,        & n=0, 2, 4, 5, 6,                      \\
 8 &  x^2=248\cdot 6^{n}+23161,      & n=0, 1, 3, 4, 5, \\
 9 &  x^2=1513\cdot 6^n+19379701008, & n=0, 7, 9, 10, 12. \\
  \end{array}
  \end{equation*}
\end{center}
\end{thm}
\begin{proof}
Although it is possible to prove our result using only elementary techniques we use the same computational approach as presented in the second part of the proof of Theorem \ref{sixsolutions}. The equations numbered 3, 5, 8, 9 were solved with the help of {\tt IntegralQuarticPoints} procedure in Magma.

(1) We consider the problem of finding all integer points on the curves $x^2=Y^3+2997a^2$ with $a=28\cdot 3^{i}$ for $i=0,1,2$.
\begin{center}
\begin{equation*}
\begin{array}{|l|l|}
  \hline
  a          & \mbox{Integral solutions}\;(Y, \pm x)\;\mbox{of}\; x^2=Y^3+2997a^2  \\
  \hline
  28          &(-108, 1044), (-63, 1449), (28, 1540), (36, 1548), (252, 4284), \\
                &(396, 8028), (441, 9387), (148932, 57475404) \\
  3\cdot 28   & (-252, 2268), (2268, 108108)\\
  9\cdot 28   & (-567, 2835), (252, 14364), (756, 24948), (5076, -361908)\\
   \hline
\end{array}
\end{equation*}
\end{center}
\begin{center}
Table 2.
\end{center}

(2) We consider the problem of finding all integer points on the curves $x^2=Y^3+414a^2$ with $a=70\cdot 3^{i}$ for $i=0,1,2$.
\begin{center}
\begin{equation*}
\begin{array}{|l|l|}
  \hline
  a          & \mbox{Integral solutions}\;(Y, \pm x)\;\mbox{of}\;x^2=Y^3+414a^2  \\
  \hline
  70          &(-126, 168), (-111, 813), (-90, 1140), (70, 1540), \\
                &(105, 1785), (210, 3360), (945, 29085), (2086, 95284), (8850, 832560)  \\
  3\cdot 70   & (-126, 4032), (225, 5445), (630, 16380)\\
  9\cdot 70   & (721, 23219), (1890, 83160), (5670, 427140), (10665, 1101465)\\
   \hline
\end{array}
\end{equation*}
\end{center}
\begin{center}
Table 3.
\end{center}

(3) We consider the problem of finding all integer points on the curves $x^2=aY^4+5550606$ with $a=130\cdot 3^{i}$ for $i=0,1,2,3$.
\begin{center}
\begin{equation*}
\begin{array}{|l|l|}
  \hline
  a           & \mbox{Integral solutions}\;(\pm Y, \pm x)\;\mbox{of}\;x^2=aY^4+5550606  \\
  \hline
  130          & (1, 2356)\\
  3\cdot 130   & \mbox{no integral solutions} \\
  9\cdot 130   & (3, 2376)\\
  27\cdot 130   & (9, 5346), (27, 43254)\\
   \hline
\end{array}
\end{equation*}
\end{center}
\begin{center}
Table 4.
\end{center}

(4) We consider the problem of finding all integer points on the curves $x^2=Y^3+41877a^2$ with $a=148\cdot 3^{i}$ for $i=0,1,2$.
\begin{center}
\begin{equation*}
\begin{array}{|l|l|}
  \hline
  a          & \mbox{Integral solutions}\;(Y, \pm x)\;\mbox{of}\;x^2=Y^3+41877a^2  \\
  \hline
  148          & (148, 30340), (1332, 57276), (3996, 254412), (8361, 765117)\\
  3\cdot 148   & \mbox{no integral solutions}\\
  9\cdot 148   & (3996, 371628), (323676, 184147668) \\
   \hline
\end{array}
\end{equation*}
\end{center}
\begin{center}
Table 5.
\end{center}

(5) We consider the problem of finding all integer points on the curves $x^2=aY^4+57402189$ with $a=8740\cdot 3^{i}$ for $i=0,1,2$.
\begin{center}
\begin{equation*}
\begin{array}{|l|l|}
  \hline
  a          & \mbox{Integral solutions}\;(\pm Y, \pm x)\;\mbox{of}\;x^2=aY^4+57402189  \\
  \hline
  8740          & (1, 7577), (3, 7623)\\
  3\cdot 8740   & (9, 15147), (2187, 774486603)\\
  9\cdot 8740   & \mbox{no integral solutions} \\
  27\cdot 8740  & (27, 354213)\\
   \hline
\end{array}
\end{equation*}
\end{center}
\begin{center}
Table 6.
\end{center}

(6) We consider the problem of finding all integer points on the curves $x^2=Y^3+11875a^2$ with $a=6\cdot 5^{i}$ for $i=0,1,2$.
\begin{center}
\begin{equation*}
\begin{array}{|l|l|}
  \hline
  a          & \mbox{Integral solutions}\;(Y, \pm x)\;\mbox{of}\;x^2=Y^3+11875a^2  \\
  \hline
  6          & (-75, 75), (-51, 543), (-50, 550), (6, 654), (61, 809), \\
                   & (150, 1950), (486, 10734), (750,  20550), (1194750, 1305916950)\\
  5\cdot 6   & (-219, 429), (-50, 3250), (150, 3750), (1125, 37875)\\
  25\cdot 6  &  (750, 26250)\\
   \hline
\end{array}
\end{equation*}
\end{center}
\begin{center}
Table 7.
\end{center}

(7) We consider the problem of finding all integer points on the curves $x^2=Y^3+6875a^2$ with $a=14\cdot 5^{i}$ for $i=0,1,2$.
\begin{center}
\begin{equation*}
\begin{array}{|l|l|}
  \hline
  a          & \mbox{Integral solutions}\;(Y, \pm x)\;\mbox{of}\;x^2=Y^3+6875a^2  \\
  \hline
  14          & (14, 1162), (350, 6650)\\
  5\cdot 14   & (-259, 4039), (-250, 4250), (350, 8750), (848750, 781933250)\\
  25\cdot 14  & (-875, 13125), (-250, 28750), (350, 29750), (1750, 78750), \\
                &(5614, 421638)\\
   \hline
\end{array}
\end{equation*}
\end{center}
\begin{center}
Table 8.
\end{center}

(8) We consider the problem of finding all integer points on the curves $x^2=aY^4+23161$ with $a=248\cdot 6^{i}$ for $i=0,1,2,3$.
\begin{center}
\begin{equation*}
\begin{array}{|l|l|}
  \hline
  a          & \mbox{Integral solutions}\;(\pm Y, \pm x)\;\mbox{of}\;x^2=aY^4+23161  \\
  \hline
  248          & (1, 153), (6, 587)\\
  6\cdot 248   & (1, 157), (6, 1397)\\
  6^2\cdot 248  & \mbox{no integral points} \\
  6^3\cdot 248  & (1,277)\\
   \hline
\end{array}
\end{equation*}
\end{center}
\begin{center}
Table 9.
\end{center}

(9) We consider the problem of finding all integer points on the curves $x^2=aY^4+19379701008$ with $a=1513\cdot 6^{i}$ for $i=0,1,2,3$.
\begin{center}
\begin{equation*}
\begin{array}{|l|l|}
  \hline
  a          & \mbox{Integral solutions}\;(\pm Y, \pm x)\;\mbox{of}\;x^2=aY^4+19379701008  \\
  \hline
  1513            & (1, 139211), (216, 1820124)\\
  6\cdot 1513     & (36, 186084)\\
  6^2\cdot 1513   & (36, 332964)\\
  6^{3}\cdot 1513 &(6, 140724)\\
   \hline
\end{array}
\end{equation*}
\end{center}
\begin{center}
Table 10.
\end{center}

\bigskip

Summing up: A quick inspection of the corresponding solution sets reveals that the only integral solutions of our equations are given in the statement of our theorem.

Our theorem is proved.
\end{proof}



\section{The equation $x^2=k^{n}+B$}\label{sec3}

In this section we are interested in finding equations of the form $x^2=k^n+B$ with "many" solutions and satisfying the condition $B\not\equiv 0\pmod{k^2}$. In order to find examples with at least three solutions we need a small modification to the method presented in Section \ref{sec2}. Indeed, here we need to assume that $A=k^{q}$ for a fixed positive integer $q<p$. Then we will have the equation $x^2=k^n+B$ with at least two solutions at $n=q, p$. In \cite[Par. 3]{Bo1} Beukers found that if
$$k=4t^2+\epsilon,\quad B=\Big(\frac{k^{m}-\epsilon}{4t}\Big)^2-k^{m},$$
where $t, m\in\N$ and $\epsilon\in\{-1,1\}$ then the Diophantine equation $x^2=k^n+B$ has at least three solutions given by
\begin{equation*}
(x,n)=\Big(\frac{k^{m}-\epsilon}{4t}-2t, 1\Big),\quad \Big(\frac{k^{m}-\epsilon}{4t}, m\Big), \quad \Big(2tk^{m}+\epsilon\frac{k^{m}-\epsilon}{4t}, 2m+1\Big).
\end{equation*}
Note that in this case $k$ is odd. A question arises whether it is possible to find similar families with $k$ even. In the next theorem we show that there are infinitely many $k$'s such that the equation $x^2=k^n+B$ has at least three solutions in positive integers and $B\not\equiv 0\pmod{k^2}$. More precisely, we have the following result:

\begin{thm}
For positive integer $t$ and a non-negative integer $m$ the Diophantine equation
\begin{equation*}
x^2=(2t)^{n}+t^2((2t)^{2(m+2)}-2(2t+1)(2t)^{m+2}+(2t-1)^2)
\end{equation*}
has at least three solutions in positive integers $(x,n)$ given by:
\begin{equation*}
(x,n)=(t((2t)^{m+2}-2t-1), 3), (t((2t)^{m+2}-2t+1), m+4), (t((2t)^{m+2}+2t-1), m+5).
\end{equation*}
\end{thm}
\begin{proof}
We left the simple check that the displayed solutions satisfy the corresponding equation to the reader.
\end{proof}

Let us observe that in the Beukers family we have $\op{gcd}(k,B)=1$. However, in our family with $k$ even this condition is not satisfied. An interesting question arises whether we can find equations $x^2=k^n+B$ having at last three solutions in positive integers with $k, B$ satisfying the condition $\gcd(B,k)=1$ with even $k$ and not of the form $2^{m}$ or with odd $k$ not of the form $4t^2+\epsilon$. It is quite interesting that we were unable to find an odd integer $k$ not of the form $4t^2+\epsilon$ and an integer $B$ such that $\op{gcd}(k,B)=1$ and the equation $x^2=k^n+B$ has at least three solutions. We considered the equations $x^2=k^n+B$ with $k\in\{7,9,11,13,19\}$ which have two solutions $n=p, q$ with $p<q<40$.  We then searched for further solution in $n\leq 100$. From the work of Bauer and Bennett \cite{Ben} we know that if $k$ is a prime number and $B$ is positive then the equation $x^2=k^n+B$ has at most three solutions satisfying $n\geq 1$. This suggest the following natural problem, which in the case $\epsilon=1$ was stated by Bauer and Bennett in \cite{Ben}.

\begin{prob}
Find an integer $B$ and an odd positive integer $k$ not of the form $4t^2+\epsilon$ with $t\in\N, \epsilon\in\{-1,1\}$, such that $\op{gcd}(k,B)=1$ and the Diophantine equation $x^2=k^n+B$ has at least three solutions in positive integers.
\end{prob}

What is going on in the case of even values of $k$? The first candidates which should be investigated are $k=6, 10, 12, ...$. Using the method described in Section \ref{sec2} we performed numerical calculations and found that if the equation $x^2=k^n+B$ has at least three solutions in non-negative integers for $k\in \{6, 10, 12, 14, 18, 20\}$ with the smallest solution $\leq 40, n\leq 100$ and $\gcd(B,k)=1$, then they are exactly three solutions of our problem. More precisely, we have the following result.

\begin{thm}
Each of the following Diophantine equations has exactly three solutions in  positive integers $(x,n)$:
\begin{equation*}
\begin{array}{lll}
 & x^2=6^{n}+2185,             & n=3, 4, 6,  \\
 & x^2=6^{n}+274837012705,     & n=4, 12, 13,\\
  & x^2=12^{n}+25029865,        & n=2, 6, 8.
 \end{array}
\end{equation*}
\end{thm}
\begin{proof}
\noindent As before we consider the problem of finding all integral points on the curve $x^2=Y^3+2185a^{2}$ with $a\in \{1,6,36\}$. Using Magma one more time we have obtained the characterizations of integral points on corresponding curves. We gather these results in the table below.
\begin{center}
\begin{tabular}{|l|l|}
  \hline
  $a$          & Integral solutions $(Y, \pm x)$ of $x^2=Y^3+2185a^2$  \\
  \hline
  $1$          & (6, 49), (36, 221), (39, 248), (156, 1949) \\
  $6$          & (36, 354)\\
  $36$         & (9, 1683), (16, 1684), (144, 2412), (864, 25452), (231336, 111266604)\\
   \hline
\end{tabular}
\end{center}
\begin{center}
Table 11.
\end{center}
A quick look into the Table 11 reveals all instances of integral points for which $Y$-th coordinate is a power of 6.

We consider now the problem of finding all integral points on the curves $x^2=aY^4+274837012705$ with $a\in\{1, 6, 36, 216\}$. The cases $a=1$ and $a=6^2$ are easy because we are dealing with equations which can be easily solved by factorization. The cases $a=6$ and $a=216$ were solved with the help of the {\tt IntegralQuarticPoints} procedure in Magma. All integral solutions of the considered equations are given below.
\begin{center}
\begin{equation*}
\begin{array}{|l|l|}
  \hline
  a          & \mbox{Integral solutions}\;(\pm Y, \pm x)\;\mbox{of}\;x^2=aY^4+274837012705  \\
  \hline
  1          & (384, 524303), (1296, 526321), (1325424, 1525921777)\\
  6          & (216, 536561)\\
  36         & \mbox{no integral solutions} \\
  216        & \mbox{no integral solutions}\\
   \hline
\end{array}
\end{equation*}
\end{center}
\begin{center}
Table 12.
\end{center}

The third equation is very easy. Indeed, if $n$ is even, say $n=2m$, we write $(x-12^m)(x+12^m)=25029865$. Considering all possible factorizations of 25029865 we get the solutions $(x,m)=(5003,1),\;(x,m)=(5293, 3)$ and  $(x,m)=(21331, 4)$. If $n$ is odd, say $n=2m+1$, then we consider the equation $x^2-12y^2=25029865$, where $y=12^{m}$. However, one can easily check this equation has no solutions mod 5. Summing up, we proved that all solutions of $x^2=12^{n}+25029865$ correspond to $n=2,6,8$.

\end{proof}

Motivated by the above examples we state the following series of conjectures concerning the solutions of certain Lebesgue-Nagell-Ramanujan type equations.

\begin{conj}
If $(x,y,n)$, where $x, y$ are positive integers and $n\geq 3$, are solutions of the Diophantine equation $x^2=y^n+2185$ then:
\begin{equation*}
\begin{array}{lll}
  n=3, &  & (x,y)=(49,6),\;(221,36),\;(248,39),\;(1949,156),\\
  n=4, &  & (x,y)=(59,6) \\
  n=6, &  & (x,y)=(221,6).
\end{array}
\end{equation*}
\end{conj}

\begin{conj}
If $(x,y,n)$, where $x, y$ are positive integers and $n\geq 3$, are solutions of the Diophantine equation $x^2=y^n+274837012705$ then:
\begin{equation*}
\begin{array}{lll}
  n=3, &  & (x,y)=(524303, 384),\;(526321, 1296),\;(1525921777, 1325424),\\
  n=4, &  & (x,y)=(526321,216) \\
  n=6, &  & (x,y)=(526321,36),\\
  n=12,&  & (x,y)=(526321,6),\\
  n=13,&  & (x,y)=(536561,6).
\end{array}
\end{equation*}
\end{conj}

\begin{conj}
If $(x,y,n)$, where $x, y$ are positive integers and $n\geq 3$, are solutions of the Diophantine equation $x^2=y^n+25029865$ then:
\begin{equation*}
\begin{array}{lll}
  n=3, &  & (x,y)=(5293, 144),\\
  n=4, &  & (x,y)=(21331, 144),\;(736181, 858), \\
  n=6, &  & (x,y)=(5293, 12),\\
  n=8, &  & (x,y)=(21331, 12).
\end{array}
\end{equation*}
\end{conj}


\section{The equation $x^2=Ak^n+B$ with $A<0$}

In this section we consider the Diophantine equation $x^2=Ak^n+B$ with constraint $A<0$. Of course, in this case the equation has only finitely many solutions in integers and we have a useful bound for solutions: each $n$ which solves the equation $x^2=Ak^n+B$ with $A<0$ satisfies $n\leq \log_{k}\left(\frac{B}{|A|}\right)$.

However, we are mainly interested in the case $A=-1$. Thus we consider the equation $x^2+k^n=B$ and ask about characterization of those positive integers $k$ such that there exists an infinite set $\cal{B}$ such that for each $B\in\cal{B}$ we have $B\not\equiv 0\pmod{k^2}$ and our equation has at least three solutions in positive integers. We propose the following result:

\begin{thm}\label{A=-1:special}
Let $t$ be an indeterminate and put $k=t^2+1$. Then for each positive integer $m$, the Diophantine equation
\begin{equation*}
x^2+k^n=\frac{k^2}{4(k-1)}\left(k^{2m+2}+2(k-2)k^{m}+1\right),
\end{equation*}
treated as an equation in polynomials, has exactly three solutions in $(x,n)\in \Z[t]\times\N$, with the solutions corresponding to $n=0,\;m+2,\;2m+2$.
\end{thm}
\begin{proof}
Let us denote the right hand side of our equation by $H_{m}(t)$. First we need to show that $H_{m}$ is indeed a polynomial with rational coefficients. In order to do this we need to check that $t=0$ is a double root of $H_{m}$. This is true for $m=0, 1$ and thus we can assume $m\geq 2$. It is clear that $t=0$ is a root of the polynomial $h_{m}(t)=k^{2m+2}+2(k-2)k^{m}+1$, where $k=k(t)=t^2+1$. Moreover, we know that $k'(0)=0$ and from the identity
$$
h_{m}'(t)=2k'(t)k(t)^{m-1}((m+1)k(t)^{m+2}+(m+1)k(t)-2m)
$$
we easily deduce that $H_{m}$ is a polynomial in $\Q[t]$. Let us observe that the following pairs $(x,n)$ are solutions of our equation:
\begin{equation*}
(x,n)=\Big(\frac{k^{m+2}+k-2}{2t}, 0\Big), \; \Big(\frac{k(k^{m+1}-1)}{2t}, m+2\Big),\; \Big(\frac{k(k^{m}(k-2)+1)}{2t},2m+2\Big).
\end{equation*}
We show now that there are no more solutions. In order to see this we note that a necessary condition for existence of a polynomial $x\in\Q[t]$ such that $x(t)^2+k^n=H_{m}(t)$ is the condition $x(1)^2+k(1)^n=H_{m}(1)=2^{2(m+1)}+1$. We are thus interested in integer solutions of the Diophantine equation $x(1)^2+2^{n}=2^{2(m+1)}+1$. Remarkably, this equation (with slightly different notation) was considered by Beukers in \cite[Theorem 2]{Bo}. He proved that the only solutions correspond to $n=0, m+2$ and $2m+2$ which are exactly our values for $n$. Our theorem is proved.

\end{proof}

It is clear that for any given integer value of $t$  we can compute $k$ and get the equation $x^2+k^n=B$ with at least three solutions in integers. The first few values of $k$ from the above theorem are $k=2,5,10,17,26,37,50,\ldots$. A question arises whether for $k$ not of the form $t^2+1$ we can find $B$ such that the equation $x^2+k^n=B$ has at least three solutions in integers. This seems to be a rather difficult question. We performed calculations for each $k<50$ not covered by the Theorem \ref{A=-1:special} in order to find integers $B$ with the property that the equation $x^2+k^n=B$ has at least three solutions in $n$ with $n\leq 50$. We collect these values together with corresponding solutions in table below. In fact, for the value of $B$ which was computed, the solutions presented in the third column of the table contain all solutions of the corresponding equation. It is interesting to note that we were unable to construct an infinite family of equations with $k\neq t^2+1$ and having at least three solutions in integers.

\begin{center}
\begin{tabular}{|l|l|l|}
  \hline
  $k$ & $B$ & values of $n$ such that $B-k^{n}$ is a square \\
   \hline
  6 & 8865               & 3,4,5 \\
  6 & 48177              & 3,5,6 \\
  6 & 2538945            & 4,7,8 \\
  6 & 334401777          & 7,9,10 \\
  6 & 1410808185         & 7,10,11 \\
  12& 448206057          & 5,7,8\\
  14& 166113185          & 4,6,7\\
  18& 4598905354020657   & 7,9,12\\
  21& 5340742            & 1,3,5\\
  22& 61234181657        & 5,7,8\\
  30& 739595025          & 3,5,6\\
  34& 170442204313460705 & 8,10,11\\
  40& 109475600          & 3,4,5\\
  40& 17264710025        & 3,5,6\\
  \hline
\end{tabular}
\end{center}
\begin{center}Table 13\end{center}


In the case of the equation $x^2=A2^{n}+B$ with $A<0$ we offer the following two conjectures.
\begin{conj}
For each positive integer $m$ the Diophantine equation
\begin{equation*}
x^2+(2^{m+1}+1)2^n=2^{4 (m+1)}+2^{3(m+1)}+2^{2 m}+2^{m+1}+1
\end{equation*}
has exactly four solutions in integers $(x,n)$ with $n=0, m+2, 2m+3, 3m+3$.
\end{conj}

\begin{conj}
For each positive integer $m$ the Diophantine equation
\begin{equation*}
x^2+\frac{1}{3}(2^{2m+6}-1)2^n=\frac{1}{9} \left(49\cdot 4^{2 m+5}-11\cdot 4^{m+3}+1\right)
\end{equation*}
has exactly four solutions in integers $(x,n)$ with $n=0, 3, 2m+7, 2m+8$.
\end{conj}

Based on our numerical experiments concerning the equation $x^2+Ak^{n}=B$ with positive $A$ and general $k$, we dare to state the following.
\begin{conj}
For any given $k\in\Z_{\geq 2}$ and positive integer $B$ the number of solutions in non-negative integers $(x,n)$ of the Diophantine equation $x^2+k^{n}=B$  is at most $3$.
\end{conj}

\begin{conj}
For any given $k\in\Z_{\geq 2}$ and positive integers $A,\;B$ the number of solutions in non-negative integers $(x,n)$ of the Diophantine equation $x^2+Ak^{n}=B$  is at most $4$.
\end{conj}

\noindent
{\bf Acknowledgments}. I am grateful to professors A. Schinzel and A. D\c{a}browski for useful remarks concerning the content of this paper. I am also grateful to the anonymous referee for constructive suggestions which improve an earlier draft of this paper.

 \bigskip

\noindent  Maciej Ulas, Jagiellonian University, Faculty of Mathematics and Computer Science, Institute of Mathematics, {\L}ojasiewicza 6, 30 - 348 Krak\'{o}w, Poland\\
e-mail:\;{\tt maciej.ulas@uj.edu.pl}

 \end{document}